\documentclass[reqno,12pt]{amsart}
\usepackage[T1]{fontenc}
\usepackage[utf8]{inputenc}
\usepackage{dsfont}
\usepackage{todonotes}
\usepackage{bm}
\usepackage{bbm}
\usepackage[nodate]{datetime}
\usepackage[hmargin=28mm,top=30mm,bottom=32mm,headheight=14pt,a4paper]{geometry}
\usepackage{color}
\usepackage{fancyhdr}
\usepackage{amsmath,amssymb,amsfonts,amsthm}
\usepackage{enumerate}
\usepackage[hidelinks]{hyperref}

\allowdisplaybreaks

\newtheorem{theorem}{Theorem}[section]
\newtheorem*{theorem*}{Theorem}
\newtheorem{lemma}[theorem]{Lemma}
\newtheorem{proposition}[theorem]{Proposition}
\newtheorem{corollary}[theorem]{Corollary}

\theoremstyle{remark}

\newtheorem{remark}[theorem]{Remark}

\newcommand{\EE}{\mathbb{E}}

\begin{document}
\title{Quantitative linear independence for square roots}
\author[M. Aymone, S. Figueredo, S. Iyer and C. T\'afula]{Marco Aymone, Samuel Figueredo, Siddharth Iyer and Christian T\'afula}

\begin{abstract}
We consider the problem of finding lower bounds for integer linear combinations of $\sqrt{a_1},\ldots,\sqrt{a_K}$, where $a_1,\ldots,a_K$ are positive integers such that their square roots are linearly independent over the rationals. We use a probabilistic approach and prove that for $K\geq 8$ and nonzero integers $m_1,\ldots,m_K$,
$$ \bigg|\sum_{n\leq K} m_n\sqrt{a_n}\bigg| > e^{\frac{2^{K-1}-1}{K} - \frac{1}{2}} \bigg(\max_{n\leq K}|m_n|\sqrt{a_n}\cdot \sqrt{K}\bigg)^{-(2^{K-1}-1)}. $$
This inequality improves the dependence on $K$ in the classical product bound.
\end{abstract}

\maketitle

\section{Introduction.}
 Let $K\geq 2$ and let $a_1,a_2,\ldots,a_K$ be positive integers such that their square-free parts are all distinct.\footnote{Every positive integer $a$ can be decomposed as $a=c^2d$, where $d$ is square-free. We say that $d$ is the square-free part of $a$.} Besicovitch \cite{besicovitch_LI} proved that their square roots are linearly independent over $\mathbb{Q}$. Our objective in this paper is to obtain quantitative bounds on this linear independence, that is, for integers $m_1,\ldots,m_K$, not all zero, we wish to prove lower bounds for 
\begin{equation}\label{equation Lambda}
\Lambda:=\bigg|\sum_{n\leq K}m_n\sqrt{a_n}\bigg|.    
\end{equation}

This question belongs to a class of difficult problems in numerical analysis and computational complexity where, for example, one wishes to obtain lower bounds for
$$\bigg|\sum_{n\leq M}\sqrt{a_n}-\sum_{M<n\leq K}\sqrt{a_n}\bigg|.$$
For a nice literature background on the computational and number-theoretic motivations, see Dubickas \cite{Dubickas_square_roots}, Steinerberger \cite{steinerberger_square_roots}, and the references therein.

For fixed $a_1,\ldots,a_K$, Eisenbrand, Haeberle and Singer \cite{eisenbrand_haeberle_singer_square_roots} proved a lower bound of the form
$$ \Lambda\geq \gamma(a_1,\ldots,a_K) \bigg(\max_{n\leq K}|m_n| \cdot K\bigg)^{-2K},$$
where $\gamma(a_1,\ldots,a_K)>0$ is an ineffective constant depending on the $a_n$. Here, we are interested in an explicit bound that is uniform in both the $a_n$ and the coefficients.

The classical effective lower bound for the linear combination $\Lambda$, due to Burnikel, Fleischer, Mehlhorn and Schirra \cite{burnikel_et_al_square_roots}, states that
\begin{equation}\label{equation Lambda Burnikel}
\Lambda\geq \bigg(\max_{n\leq K}|m_n|\sqrt{a_n}\cdot K\bigg)^{-(2^{K}-1)}.    
\end{equation}
Their proof uses the fact that, for $\varepsilon=(\varepsilon_1,\ldots,\varepsilon_K)\in\{-1,1\}^K$
$$P := \prod_{\varepsilon\in\{-1,1\}^K}\bigg(\sum_{n\leq K}\varepsilon_n m_n\sqrt{a_n}\bigg)$$
is a non-vanishing integer; see Lemma \ref{lemma X} below.

We re-examine this argument. Applying the AM-GM inequality to $P^2$, and then interpreting $(\varepsilon_n)_{n\leq K}$ as i.i.d. Rademacher random variables, orthogonality gives
$$ \EE\bigg|\sum_{n\leq K}\varepsilon_n m_n\sqrt{a_n}\bigg|^2 = \sum_{n\leq K}m_n^2 a_n \leq K\bigg(\max_{n\leq K}|m_n|\sqrt{a_n}\bigg)^2. $$

This reasoning already gives a saving in the power of $K$ in comparison with \eqref{equation Lambda Burnikel} (see Proposition \ref{old thm}), and suggests that further improvements can be made by a finer analysis of the discrepancy between the first and second moments of Rademacher sums. On the one hand, Cauchy--Schwarz applied to Rademacher sums can not be improved for general weights, since one single weight can concentrate all mass of the sum. On the other hand, when the weights of these Rademacher sums have comparable size, then a saving can be obtained -- here we refer to these recent papers \cite{gao_qian} by Gao and Qian, and to \cite{Jakimiuk_et_al} by Jakimiuk, Tang and Tkocz and the references therein for the literature on this field of research.

Here in our paper, we can divide our problem of finding lower bounds for $\Lambda$ into two cases: one is that $\Lambda$ is distant from $0$, and hence that the proof should be straightforward if this distance is relatively large, and the other in which $\Lambda$ is close to 0. In the second and more interesting case, $\Lambda$ being small means heuristically that the weights $m_n\sqrt{a_n}$ are not too sparse, and hence that a saving can be obtained in a Cauchy--Schwarz argument for Rademacher sums akin \cite{gao_qian} and \cite{Jakimiuk_et_al}. Indeed, by inserting the fourth and sixth moments on this approach, we obtain the following further improvement upon the $K$-dependence of \eqref{equation Lambda Burnikel}.

\begin{theorem}\label{theorem principal}
Let $K\geq 8$ and let $a_1,a_2,\ldots,a_K$ be positive integers such that their square-free parts are all distinct. Let $m_1,\ldots,m_K$ be nonzero integers, and let $\Lambda$ be as in \eqref{equation Lambda}. Then, we have that
\begin{equation}\label{equation main bound}
\Lambda\geq e^{(2^{K-1}-1)/K} \bigg(\frac{1}{1-2^{1-K}}\sum_{n\leq K} m_n^2a_n\bigg)^{-(2^{K-1}-1)/2}.
\end{equation}
In particular,
$$ \Lambda\geq C_K\, e^{(2^{K-1}-1)/K} \bigg(\max_{n\leq K}|m_n|\sqrt{a_n}\cdot \sqrt{K}\bigg)^{-(2^{K-1}-1)}, $$
where $C_K := (1-2^{1-K})^{(2^{K-1}-1)/2}>e^{-1/2}$. Moreover, $C_K=e^{-1/2}+O(2^{-K})$.
\end{theorem}

If $K=1$, then $\Lambda\geq 1$. For $2\leq K\leq7$, the bound without the factor $e^{(2^{K-1}-1)/K}$ in \eqref{equation main bound} remains valid; see Proposition \ref{old thm}. The exponent in the factor $e^{(2^{K-1}-1)/K}$ is asymptotically optimal for our method; see Remark \ref{remark sharpness}.

In the case of
$$ \bigg|\sum_{n\leq M}\sqrt{a_n}-\sum_{M<n\leq K}\sqrt{a_n}\bigg|, $$
where $1\leq M<K$ and all $a_n\leq N$, Dubickas \cite[Theorem 1]{Dubickas_square_roots} proved that, whenever the two sums are distinct,
$$ \bigg|\sum_{n\leq M}\sqrt{a_n}-\sum_{M<n\leq K}\sqrt{a_n}\bigg|\geq \frac{1}{(K^2 N)^{2^{K-2}-1/2}}. $$
This estimate does not require the square-free parts of the $a_n$ to be distinct. Under the independence hypothesis of Theorem \ref{theorem principal}, we obtain the following improvement on the $K$-dependence.

\begin{corollary}
Under the hypotheses on $K$ and $a_1,\ldots,a_K$ of Theorem \ref{theorem principal}, if for all $n\leq K$ we have $a_n\leq N$, then, for any $1\leq M\leq K$ we have that
$$ \bigg|\sum_{n\leq M}\sqrt{a_n}-\sum_{M<n\leq K}\sqrt{a_n}\bigg|\geq \frac{C_K\, e^{(2^{K-1}-1)/K}}{(KN)^{2^{K-2}-1/2}}. $$
\end{corollary}

This corollary is an immediate consequence of Theorem \ref{theorem principal} upon noticing that $|m_n|=1$ for all $n\leq K$ and hence $\sum_{n\leq K}m_n^2a_n\leq KN$.

\section{Lemmas}\label{section lemmas}

We begin by revisiting one result of \cite{burnikel_et_al_square_roots}. We provide a short elementary proof for the following result that does not require Galois theory.

\begin{lemma}\label{lemma X}
Let $K\geq 1$ and let $a_1,\ldots,a_K$ be positive integers with distinct
square-free parts. Let $m_1,\ldots,m_K$ be integers, not all zero, and $\varepsilon=(\varepsilon_1,\ldots,\varepsilon_K)\in\{-1,1\}^K$. Then
$$ P=\prod_{\varepsilon\in\{-1,1\}^K}\bigg(\sum_{n\leq K}\varepsilon_n m_n\sqrt{a_n}\bigg)\in\mathbb{Z}\setminus\{0\}. $$
\end{lemma}
\begin{proof}
Consider the polynomial
$$ f(X_1,\ldots,X_K) := \prod_{\varepsilon\in\{-1,1\}^K} \Bigg(\sum_{n\leq K}\varepsilon_n X_n\Bigg) \in \mathbb{Z}[X_1,\ldots,X_K]. $$
Fix an index $j\leq K$. For each choice of the signs $(\varepsilon_n)_{n\neq j}$, the product contains the two factors
$$ \Bigg(\sum_{\substack{n\leq K\\ n\neq j}} \varepsilon_n X_n\Bigg) + X_j \qquad\text{and}\qquad
\Bigg(\sum_{\substack{n\leq K\\ n\neq j}} \varepsilon_n X_n\Bigg) - X_j,$$
corresponding to $\varepsilon_j = 1$ and $\varepsilon_j = -1$. Therefore
$$ f(X_1,\ldots,X_K) = \prod_{\substack{\varepsilon\in\{-1,1\}^K \\ \varepsilon_j=1}} \Bigg(\Bigg(\sum_{\substack{n\leq K\\n\neq j}} \varepsilon_n X_n\Bigg)^2 - X_j^2\Bigg), $$
so $f$ contains only even powers of $X_j$. Since $j$ was arbitrary, every variable occurs to an even power. Thus, there exists $g\in \mathbb{Z}[X_1,\ldots,X_K]$ such that
$$ f(X_1,\ldots,X_K) = g(X_1^2,\ldots,X_K^2). $$
Substituting $X_n = m_n\sqrt{a_n}$, we obtain $P = g(m_1^2 a_1,\ldots, m_K^2 a_K)\in \mathbb{Z}$.

Finally, by Besicovitch's \cite{besicovitch_LI} linear independence of the square roots $(\sqrt{a_n})_{n\leq K}$, we have $\sum_{n\leq K}\varepsilon_n m_n\sqrt{a_n}\neq 0$ for every choice of signs. Hence $P\neq 0$.
\end{proof}

Using only the second moment of a Rademacher sum, we can already obtain the following improvement upon the classical bound.

\begin{proposition}[AM-GM bound]\label{old thm}
Let $K\geq2$ and let $a_1,\ldots,a_K$ be positive integers with distinct square-free parts. For integers $m_1,\ldots,m_K$, not all zero, we have
$$ \Lambda\geq\bigg(\frac{1}{1-2^{1-K}}\sum_{n\leq K}m_n^2a_n\bigg)^{-(2^{K-1}-1)/2}. $$
In particular,
$$ \Lambda\geq C_K\bigg(\max_{n\leq K}|m_n|\sqrt{a_n}\cdot\sqrt K\bigg)^{-(2^{K-1}-1)}, $$
where $C_K=(1-2^{1-K})^{(2^{K-1}-1)/2}$.
\end{proposition}
\begin{proof}
By Lemma \ref{lemma X}, isolating the two factors of absolute value $\Lambda$, we have
\begin{equation}\label{equation product}
1\leq\Lambda^2\prod_{\substack{\varepsilon\in\{-1,1\}^K\\ \varepsilon\neq\pm(1,\ldots,1)}}\bigg|\sum_{n\leq K}\varepsilon_nm_n\sqrt{a_n}\bigg|.
\end{equation}
Applying AM-GM to the squares of the remaining $2^K-2$ factors, we obtain
$$ 1 \leq\Lambda^2\bigg(\frac{1}{2^K-2}\sum_{\varepsilon\in\{-1,1\}^K}\bigg|\sum_{n\leq K}\varepsilon_n m_n\sqrt{a_n}\bigg|^2\bigg)^{2^{K-1}-1}. $$
By interpreting $(\varepsilon_n)_{n\leq K}$ as i.i.d. Rademacher random variables and using that $\EE(\varepsilon_i\varepsilon_j)=0$ for $i\neq j$, it follows that
$$ \frac{1}{2^K}\sum_{\varepsilon\in\{-1,1\}^K}\bigg|\sum_{n\leq K}\varepsilon_nm_n\sqrt{a_n}\bigg|^2
=\EE\bigg|\sum_{n\leq K}\varepsilon_nm_n\sqrt{a_n}\bigg|^2
=\sum_{n\leq K}m_n^2a_n. $$
Hence
$$ 1\leq \Lambda^2\bigg(\frac{1}{1-2^{1-K}}\sum_{n\leq K}m_n^2a_n\bigg)^{2^{K-1}-1}. $$
Rearranging proves the first bound. The conclusion with the maximum follows from $\sum_{n\leq K}m_n^2a_n\leq K(\max_{n\leq K}|m_n|\sqrt{a_n})^2$.
\end{proof}

\section{Moment estimates}

To sharpen the previous estimate, we will use the fourth and sixth moments. The following identities also appear, in normalized form, in \cite[(1) and (15)]{gao_qian}.

\begin{lemma}\label{lemma moments}
Let $K\geq 1$ and let $X = \sum_{n\leq K} \varepsilon_n x_n$, where the $\varepsilon_n$ are i.i.d. Rademacher random variables and the $x_n$ are real coefficients, and write $S=\sum_{n\leq K}x_n^2$. Then
\begin{align*}
\EE(X^2)&=S,\\
\EE(X^4)&=3S^2-2\sum_{n\leq K}x_n^4,\\
\EE(X^6)&=15S^3-30S\sum_{n\leq K}x_n^4+16\sum_{n\leq K}x_n^6.
\end{align*}
\end{lemma}
\begin{proof}
For nonnegative integers $r_1,\ldots,r_K$, independence gives
$$ \EE\bigg(\prod_{n\leq K}\varepsilon_n^{r_n}\bigg)
=\prod_{n\leq K}\frac{1+(-1)^{r_n}}{2}
=\begin{cases}
1,&\text{if every $r_n$ is even},\\
0,&\text{otherwise}.
\end{cases} $$
In particular, $\EE(\varepsilon_i\varepsilon_j)=0$ for $i\neq j$ and $\EE(\varepsilon_i^2)=1$, giving $\EE(X^2)=S$. Expanding the fourth and sixth powers and applying the same identity, we obtain
\begin{align*}
\EE(X^4)&=\sum_nx_n^4+\frac{4!}{2!2!}\sum_{i<j}x_i^2x_j^2,\\
\EE(X^6)&=\sum_nx_n^6+\frac{6!}{4!2!}\sum_{i\neq j}x_i^4x_j^2
+\frac{6!}{(2!)^3}\sum_{i<j<k}x_i^2x_j^2x_k^2.
\end{align*}
The fourth moment formula follows from $S^2=\sum_nx_n^4+2\sum_{i<j}x_i^2x_j^2$. For the sixth moment, we use
\begin{align*}
S\sum_nx_n^4&=\sum_nx_n^6+\sum_{i\neq j}x_i^4x_j^2,\\
S^3&=\sum_nx_n^6+3\sum_{i\neq j}x_i^4x_j^2
+6\sum_{i<j<k}x_i^2x_j^2x_k^2.
\end{align*}
Substituting these into the expansion of $\EE(X^6)$ gives the result.
\end{proof}

For a positive random variable $Y$ with $\EE(Y)=1$, the AM-GM inequality gives $\EE(\log Y) \leq \log(\EE(Y))=0$. The following lemma sharpens this bound using the second and third moments of $Y$.

\begin{lemma}\label{lemma geometric mean}
Let $Y > 0$ be a random variable with $\EE(Y)=1$ and $\EE(Y^3)<\infty$. Then
$$ \EE(\log Y)\leq-\frac{8}{9}\frac{(\EE(Y^2)-1)^2}{\EE(Y^3)-\EE(Y^2)}, $$
where the quotient is interpreted as $0$ when $Y=1$ almost surely.
\end{lemma}
\begin{proof}
For $u>0$ we have
$$ \log u\leq u-1-\frac{8}{9}\frac{(u-1)^2}{u+1}. $$
Indeed, by subtracting the left-hand side from the right-hand side, this difference vanishes at $u=1$, and has derivative $(u-1)(u-3)^2/(9u(u+1)^2)$. This derivative is nonpositive for $u<1$ and nonnegative for $u>1$, so the difference attains its minimum at $u=1$.

If $Y=1$ almost surely, then the lemma holds with equality, so assume otherwise. Then $\EE((Y-1)^2(Y+1))>0$. Since $\EE(Y-1)=0$, taking expectations and applying Cauchy--Schwarz gives
\begin{align*}
-\EE(\log Y) \geq\frac{8}{9}\,\EE\bigg(\frac{(Y-1)^2}{Y+1}\bigg) &= \frac{8}{9}\,\EE\bigg(\bigg(\frac{Y-1}{\sqrt{Y+1}}\bigg)^2\bigg) \\
&\geq\frac{8}{9} \frac{\left(\EE\left(\dfrac{Y-1}{\sqrt{Y+1}}\cdot(Y-1)\sqrt{Y+1}\right)\right)^2}{\EE\big(((Y-1)\sqrt{Y+1})^2\big)}\\
&=\frac{8}{9} \frac{(\EE((Y-1)^2))^2}{\EE((Y-1)^2(Y+1))}.
\end{align*}
Substituting $\EE((Y-1)^2)=\EE(Y^2)-1$ and $\EE((Y-1)^2(Y+1)) = \EE(Y^3)-\EE(Y^2)$ completes the proof.
\end{proof}

\section{Improving the AM-GM bound}\label{section amgm}

For this section, let $x_1,\ldots,x_K$ be real numbers with $|x_n|\geq1$, write
$$ X=\sum_{n\leq K}\varepsilon_n x_n,\qquad S=\sum_{n\leq K}x_n^2,\qquad L=\bigg|\sum_{n\leq K}x_n\bigg|, $$
where the $\varepsilon_n$ are i.i.d. Rademacher random variables, and assume that $X\neq0$ for every choice of signs. We denote by $\EE'$ the average over the $2^K-2$ choices other than $(1,\ldots,1)$ and $(-1,\ldots,-1)$. Since $X=\pm L$ at the two omitted choices, for $r=1,2,3$ we have that
\begin{equation}\label{equation conditional moments}
\EE'(X^{2r}) = \frac{\EE(X^{2r})-2^{1-K}L^{2r}}{1-2^{1-K}}.
\end{equation}

We consider the cases $\max_nx_n^2\leq S/2$ and $\max_nx_n^2>S/2$ separately. In the first case, the fourth and sixth moment formulas give the estimates needed to apply Lemma \ref{lemma geometric mean} to $X^2/\EE'(X^2)$, with respect to the average $\EE'$. This gives the following improvement upon AM-GM.

\begin{proposition}\label{prop small coordinates}
If $K\geq8$, $L<1$, and $\max_nx_n^2\leq S/2$, then
$$ \EE'(\log|X|) < \frac{1}{2}\log\bigg(\frac{S}{1-2^{1-K}}\bigg)-\frac{1}{8}. $$
\end{proposition}
\begin{proof}
By hypothesis we have that $\sum_nx_n^6\leq(S/2)\sum_nx_n^4$. Now, Lemma \ref{lemma moments} gives
\begin{align*}
\EE(X^6)&=15S^3-30S\sum_nx_n^4+16\sum_nx_n^6\\
&\leq 15S^3-22S\sum_n x_n^4
=11S\,\EE(X^4)-18S^3.
\end{align*}
By substituting this last inequality into \eqref{equation conditional moments}, and expressing $\EE X^4$ in terms of $\EE'X^4$, and using that $L^2<1\leq S$ and hence $-L^6+11SL^4-18S^3<-L^6+11S-18S^2\leq -7<0$, we reach at
\begin{equation}\label{moment inequality}
 \EE'(X^6)\leq 11S\,\EE'(X^4)-18S^3. 
\end{equation}
Put $\mu := \EE'(X^2)$ and $Y := X^2/\mu$, so that $\EE'(Y)=1$. By \eqref{equation conditional moments} and the assumption $L< S$, we have that
$$ S\leq\mu\leq\frac{S}{1-2^{1-K}}. $$
Dividing \eqref{moment inequality} by $\mu^3$, we obtain
\begin{align*}
\EE'(Y^3) &\leq 11\frac{S}{\mu}\,\EE'(Y^2)-18\bigg(\frac{S}{\mu}\bigg)^3\\
&\leq 11\,\EE'(Y^2)-18(1-2^{1-K})^3
\leq 11\,\EE'(Y^2)-\frac{35}{2},
\end{align*}
since $K\geq 8$ and $18\,(127/128)^3>35/2$. Moreover,
\begin{align*}
\EE'(Y^3)-\EE'(Y^2) &\leq 10(\EE'(Y^2)-1)-\frac{15}{2}\\
&=\frac{10}{3}(\EE'(Y^2)-1)^2-\frac{10}{3}\bigg(\EE'(Y^2)-\frac{5}{2}\bigg)^2\\
&\leq\frac{10}{3}(\EE'(Y^2)-1)^2.
\end{align*}
Applying Lemma \ref{lemma geometric mean} to the average $\EE'$, it follows that
$$ \EE'(\log Y)\leq-\frac{8}{9}\frac{(\EE'(Y^2)-1)^2}{\EE'(Y^3)-\EE'(Y^2)}\leq-\frac{4}{15}<-\frac{1}{4}. $$
Therefore
\[ \EE'(\log|X|)=\frac{1}{2}\log\mu+\frac{1}{2}\,\EE'(\log Y) < \frac{1}{2}\log\bigg(\frac{S}{1-2^{1-K}}\bigg)-\frac{1}{8}. \qedhere \]
\end{proof}

When $\max_n x_n^2\geq S/2$, we pair the two choices of the sign of a largest coefficient before applying AM-GM. Before doing that, we first establish the precise inequality needed for the estimate.

\begin{lemma}\label{lemma scalar comparison}
If $K\geq 8$ and $1/(K+2^{3-K})\leq t\leq1/2$, then
$$ \log(2+2^{1-K})+\frac{2^{K-1}-2}{4} \log\bigg(1-4t+\bigg(6-\frac2{K-1}\bigg)t^2\bigg) < -\frac{2^{K-1}-1}{K}. $$
\end{lemma}
\begin{proof}
Put $b=1/(K+2^{3-K})$ and
$$ f(t)=1-4t+\bigg(6-\frac2{K-1}\bigg)t^2. $$
Since $f''(t)>0$, the maximum of $f$ on $[b,1/2]$ occurs at an endpoint. Moreover,
$$ f(b) \geq 1-4b > \frac{1}{2} > f(1/2), $$
so to prove the inequality it suffices to consider $t=b$.

Suppose first that $K\geq10$. By Cauchy--Schwarz, for $0<v<1$,
$$ -\log v=\int_v^1\frac{\mathrm{d}u}{u} \geq\frac{\big(\int_v^1 \mathrm{d}u\big)^2}{\int_v^1 u\,\mathrm{d}u} = \frac{2(1-v)}{1+v}. $$
Since $0 < f(b) \leq 1-4b+6b^2<1$, this gives
\begin{align*}
 -\frac{1}{4}\log f(b) \geq -\frac{1}{4}\log(1-4b+6b^2) &\geq \frac{b(2-3b)}{2-4b+6b^2} \\
 &> \frac{b}{1-b/5},
\end{align*}
where the last inequality holds because for $b\leq1/10<1/9$ we have that
$$(2-3b)\bigg(1-\frac{b}{5}\bigg)>2-4b+6b^2.$$
Thus, as $2^{3-K} < 1/30$,
$$ -\frac{1}{4}\log f(b) > \frac{1}{K+2^{3-K}-1/5} > \frac{1}{K-1/6} \geq\frac{1}{K}+\frac1{6K^2}. $$
On the other hand, as $\log(2 + 1/512) < 3/4$,
$$ \log(2+2^{1-K})+\frac{1}{K} < \frac{3}{4}+\frac{1}{10} = \frac{17}{20} \leq\frac{2^{K-1}-2}{6K^2}. $$
The last expression equals $17/20$ at $K=10$, and is increasing for $K\geq 10$. Combining these two estimates proves the inequality.

For $K=8,9$, respectively, it remains to check
\begin{align*}
-\frac{63}{2}\log f(32/257) &>\frac{127}{8}+\log(257/128),\\
-\frac{127}{2}\log f(64/577) &>\frac{85}{3}+\log(513/256).
\end{align*}
Here $f$ is evaluated with the corresponding value of $K$. Both inequalities can be checked computationally.
\end{proof}

We can now prove the estimate for $\max_n x_n^2\geq S/2$.

\begin{proposition}\label{prop large coordinate}
If $K\geq 8$, $L<2^{1-K}$, and $\max_nx_n^2\geq S/2$, then
$$ \EE'(\log|X|) < \frac{1}{2}\log\bigg(\frac{S}{1-2^{1-K}}\bigg)-\frac{1}{K}. $$
\end{proposition}
\begin{proof}
Choose $j$ with $|x_j|=\max_n|x_n|$, put $Z = \sum_{n\neq j}\varepsilon_n x_n$, and write $t := \EE(Z^2)/S$. Since $x_j=\sum_nx_n-\sum_{n\neq j}x_n$, the triangle inequality and Cauchy--Schwarz give
$$ |x_j|\leq L+\bigg|\sum_{n\neq j}x_n\bigg|
\leq L+\sqrt{K-1}\sqrt{\EE(Z^2)}. $$
Using $\EE(Z^2)=\sum_{n\neq j}x_n^2\geq K-1$ and $L<2^{1-K}$, we get
$$ \frac{|x_j|}{\sqrt{\EE(Z^2)}}
\leq\sqrt{K-1}+\frac{L}{\sqrt{\EE(Z^2)}}
<\sqrt{K-1}+\frac{2^{1-K}}{\sqrt{K-1}}. $$
Consequently,
\begin{align*}
\frac1t=\frac{S}{\EE(Z^2)}=1+\frac{x_j^2}{\EE(Z^2)}
&<1+\bigg(\sqrt{K-1}+\frac{2^{1-K}}{\sqrt{K-1}}\bigg)^2\\
&=K+2^{2-K}+\frac{2^{2-2K}}{K-1}<K+2^{3-K},
\end{align*}
where the last inequality uses $2^{-K}<K-1$. Also, since $x_j^2\geq S/2$, we have that $t=(S-x_j^2)/S\leq 1/2$. Thus
\begin{equation}\label{equation variance fraction}
\frac1{K+2^{3-K}}<t\leq\frac{1}{2}.
\end{equation}

We now estimate 
$$\EE'(\log|X|)=\frac{1}{2^K-2}\sum_{\substack{(\varepsilon_n)_{n\leq K}\\ \text{not all equal}}} \log|\varepsilon_jx_j+Z|.$$ 
Since $\EE'$ excludes the sign choices $(1,\ldots,1)$ and $(-1,\ldots,-1)$, if all signs in $Z$ are $1$, the only possible choice for $\varepsilon_j$ in the sum above is $-1$; if they are all $-1$, then the only possibility for $\varepsilon_j$ is $+1$. Now we estimate the contribution of $\log |X|$ in this two extreme situations. We first notice that, in both cases,  $|X|=|\sum_{n\leq K} x_n-2x_j|$, and since $L<2^{1-K}$, $|x_j|\leq\sqrt{S}$, and $S\geq1$, this can be bounded by
$$ \bigg|\sum_{n\leq K} x_n-2x_j\bigg|\leq L+2|x_j|\leq(2+2^{1-K})\sqrt{S}. $$
For each of the other $2^{K-1}-2$ choices of signs in $Z$, there is a one-to-one correspondence between these sequences $(\varepsilon_n)_{n}$ with $\varepsilon_j=1$ and with $\varepsilon_j=-1$, and each of them has same probability. By pairing each of these flips in the $j$-th random variable, the factors $|Z+x_j|$ and $|Z-x_j|$ contribute to a factor of $|Z^2-x_j^2|$. Since all summands are nonnegative,
$$ \sum_{\substack{(\varepsilon_n)_{n\neq j}\\ \text{not all equal}}}(x_j^2-Z^2)^2
\leq\sum_{(\varepsilon_n)_{n\neq j}}(x_j^2-Z^2)^2
=2^{K-1}\EE\big((x_j^2-Z^2)^2\big). $$
Applying AM-GM to these squares and taking logarithms gives
\begin{align}
(2^K-2)\,\EE'(\log|X|) &=2\log\bigg|\sum_{n\leq K} x_n-2x_j\bigg| +\frac{1}{2}\sum_{\substack{(\varepsilon_n)_{n\neq j}\\ \text{not all equal}}} \log\big((x_j^2-Z^2)^2\big) \nonumber \\
&\leq 2\log\big((2+2^{1-K})\sqrt{S}\big) +\frac{2^{K-1}-2}{2} \log\bigg(\frac{2^{K-1}\,\EE\big((x_j^2-Z^2)^2\big)}{2^{K-1}-2}\bigg). \label{equation moment bound}
\end{align}

By Lemma \ref{lemma moments}, we have that
$$ \EE\big((x_j^2-Z^2)^2\big) = x_j^4-2x_j^2\,\EE(Z^2)+3\,(\EE(Z^2))^2-2\sum_{n\neq j}x_n^4. $$
As $t = \EE(Z^2)/S$, we have $x_j^2=(1-t)S$. Since $\EE(Z^2)\leq\sqrt{(K-1)\sum_{n\neq j} x_n^4}$ by Cauchy--Schwarz, the previous equality implies that
\begin{align*}
\EE\big((x_j^2-Z^2)^2\big) &\leq x_j^4- 2x_j^2\,\EE(Z^2)+\bigg(3-\frac{2}{K-1}\bigg)\,\EE(Z^2)^2\\
&=S^2\bigg(1-4t+\bigg(6-\frac{2}{K-1}\bigg)t^2\bigg).
\end{align*}
Substituting in \eqref{equation moment bound} and dividing by $2^K-2$, we obtain
\begin{align*}
\EE'(\log|X|) &\leq \frac{1}{2}\log S \\
&+\frac{1}{2^{K-1}-1}\bigg(\log(2+2^{1-K}) +\frac{2^{K-1}-2}{4} \log\bigg(1-4t+\bigg(6-\frac{2}{K-1}\bigg)t^2\bigg)\bigg) \\
&-\frac{2^{K-1}-2}{4(2^{K-1}-1)}\log(1-2^{2-K}).
\end{align*}
Thus, by \eqref{equation variance fraction} and Lemma \ref{lemma scalar comparison},
$$ \EE'(\log|X|)<\frac{1}{2}\log S-\frac{1}{K} -\frac{2^{K-1}-2}{4(2^{K-1}-1)}\log(1-2^{2-K}). $$
Convexity\footnote{If $f''\geq 0$ in $[a,b]$, then $f((a+b)/2) \leq (f(a)+f(b))/2$.} of $u\mapsto u\log u$, applied at the midpoint of $1$ and $1-2^{2-K}$, gives
$$ (1-2^{1-K})\log(1-2^{1-K}) \leq\frac{1}{2}(1-2^{2-K})\log(1-2^{2-K}). $$
Therefore
\[ \EE'(\log|X|) < \frac{1}{2}\log\bigg(\frac{S}{1-2^{1-K}}\bigg)-\frac{1}{K}. \qedhere \]
\end{proof}

\section{Proof of Theorem \ref{theorem principal}}\label{section proof}

Take $x_n=m_n\sqrt{a_n}$ in the results of Section \ref{section amgm}. Then $|x_n|\geq1$, $S=\sum_{n\leq K}m_n^2a_n$, and $L=\Lambda$. By Lemma \ref{lemma X}, $\sum_n \varepsilon_n x_n\neq 0$ for every choice of signs.

If $\Lambda\geq 2^{1-K}$, the theorem follows from $S\geq K\geq8$ and $e^{-2/K}\geq 1-2/K$, since
$$ \bigg(\frac{e^{-2/K}S}{1-2^{1-K}}\bigg)^{-(2^{K-1}-1)/2}< (K-2)^{-(2^{K-1}-1)/2} < 2^{1-K}. $$

Suppose now that $\Lambda<2^{1-K}$. If $\max_nx_n^2\leq S/2$, we apply Proposition \ref{prop small coordinates} and use $1/8\geq1/K$; if $\max_nx_n^2\geq S/2$, we apply Proposition \ref{prop large coordinate}. In either case,
$$ \EE'(\log|X|) < \frac{1}{2}\log\bigg(\frac{S}{1-2^{1-K}}\bigg)-\frac{1}{K}. $$
Since, by Lemma \ref{lemma X}, $P$ is a nonzero integer, taking logarithms gives
\begin{align*}
0\leq\log|P|&=2\log\Lambda + (2^K-2)\,\EE'(\log|X|)\\
&\leq2\log\Lambda + (2^{K-1}-1)\log\bigg(\frac{e^{-2/K}S}{1-2^{1-K}}\bigg).
\end{align*}
Rearranging proves \eqref{equation main bound}, and the conclusion follows from $S\leq K(\max_{n\leq K}|m_n|\sqrt{a_n})^2$.

Finally, applying $\log(1-u)>-u/(1-u)$ at $u=2^{1-K}$ gives $C_K>e^{-1/2}$. The asymptotic follows by expanding the logarithm:
\[ \log C_K=\frac{2^{K-1}-1}{2}\log(1-2^{1-K}) = -\frac{1}{2} + O(2^{-K}). \]
This completes the proof. \hfill$\square$

\section{Remarks}\label{section remarks}

\begin{remark}[Optimality]\label{remark sharpness}
For real coefficients $x_n$ satisfying only the assumptions of Section \ref{section amgm} and $Q := \prod_{\varepsilon\in\{-1,1\}^K}(\sum_{n\leq K}\varepsilon_n x_n)\in\mathbb{Z}\setminus\{0\}$, the factor $e^{(2^{K-1}-1)/K}$ in the lower bound for $L$, as it appears in \eqref{equation main bound}, cannot be replaced by $e^{c(2^{K-1}-1)/K}$ for any fixed $c>1$. To see this, choose $\eta>0$ such that
$$ \eta\prod_{h=1}^{K-1}(2h+\eta)^{\binom{K-1}{h}}=1. $$
The solution exists since the left-hand side increases continuously from $0$ to infinity, and satisfies $0<\eta<2^{-(2^{K-1}-1)}$. Take $(x_1,\ldots,x_K)=(K-1+\eta,-1,\ldots,-1)$. For $0\leq h\leq K-1$, the value $2h+\eta$ occurs $2\binom{K-1}{h}$ times among the absolute values of the signed sums. So $|Q|=1$, $L=\eta<2^{1-K}$, and
\begin{equation}
 \EE'(\log|X|) = \frac{1}{2^{K-1}-1}\sum_{h=1}^{K-1}\binom{K-1}{h}\log(2h) + O(\eta), \label{eq logX}
\end{equation}
where we used that $0\leq\log(2h+\eta)-\log(2h)\leq\eta/2$.

Taylor expansion with Cauchy remainder gives, uniformly for $1\leq h\leq K-1$,
\begin{align*}
\log(2h) = \log(K-1) &+\frac{2h-K+1}{K-1} - \frac{(2h-K+1)^2}{2(K-1)^2} \\
&+\frac{(2h-K+1)^3}{3(K-1)^3} + O\bigg(\frac{(2h-K+1)^4}{h(K-1)^3}\bigg).
\end{align*}
By the symmetry $h\mapsto K-1-h$, for $r=1,3$ we have 
$$ \frac{1}{2^{K-1}-1}\sum_{h=1}^{K-1}\binom{K-1}{h}\bigg(\frac{2h-K+1}{K-1}\bigg)^r = \frac{1}{2^{K-1}-1} = O(2^{-K}). $$
The sum $\sum_{n\leq K} \varepsilon_n$ equals $2h-K$ for exactly $\binom{K}{h}$ choices of $(\varepsilon_n)_{n\leq K}$. Thus, by the second moment identity
$$ \sum_{h=0}^{K-1}\binom{K-1}{h} (2h-K+1)^2 = 2^{K-1}\,\EE\bigg(\bigg(\sum_{n\leq K-1}\varepsilon_n\bigg)^2\bigg) = 2^{K-1}(K-1), $$
the quadratic term in \eqref{eq logX} averages to $-1/(2(K-1)) + O(2^{-K})$. For the remainder term, using $\binom{K-1}{h}/h=\frac{h+1}{Kh}\binom{K}{h+1}\leq\frac{2}{K}\binom{K}{h+1}$, the fourth moment in Lemma \ref{lemma moments} gives
\begin{align*}
\sum_{h=1}^{K-1}\binom{K-1}{h}\frac{(2h-K+1)^4}{h}
&\leq\frac{2}{K}\sum_{h=0}^{K}\binom{K}{h}(2h-K-1)^4\\
&=\frac{2^{K+1}}{K}\,\EE\bigg(\bigg(\sum_{n\leq K}\varepsilon_n-1\bigg)^4\bigg) \\
&=\frac{2^{K+1}}{K}(3K^2+4K+1)
=O(2^K K).
\end{align*}
Dividing by $(2^{K-1}-1)(K-1)^3$, the averaged remainder contributes $O(K^{-2})$. Therefore, substituting it all in \eqref{eq logX}, we obtain
\begin{equation}
 \EE'(\log|X|) = \log(K-1)-\frac{1}{2(K-1)}+O(K^{-2}). \label{eq logX2}
\end{equation}

On the other hand, since $S=(K-1+\eta)^2+(K-1) = (1+O(\eta/K))\,K(K-1)$ and $\eta < 2^{-(2^{K-1}-1)}$, we have
\begin{align*}
\frac{1}{2}\log\bigg(\frac{S}{1-2^{1-K}}\bigg) &=\log(K-1)+\frac{1}{2}\log\bigg(1+\frac1{K-1}\bigg) -\frac{1}{2}\log(1-2^{1-K})+O(2^{-K})\\
&=\log(K-1)+\frac{1}{2(K-1)}+O(K^{-2}),
\end{align*}
where we used $\log(1+u)=u+O(u^2)$. Moreover, since $|Q|=1$, we have $\log L = -(2^{K-1}-1)\,\EE'(\log|X|)$. Together with \eqref{eq logX2}, this gives
$$ L\,\bigg(\frac{S}{1-2^{1-K}}\bigg)^{(2^{K-1}-1)/2} = \exp\bigg(\frac{2^{K-1}-1}{K}+O\bigg(\frac{2^K}{K^2}\bigg)\bigg). $$
Consequently, the leading term $(2^{K-1}-1)/K$ in the exponent cannot be improved under the assumptions of our method. This does not rule out lower order improvements in the exponent, nor stronger bounds using the special form $x_n=m_n\sqrt{a_n}$.
\end{remark}

\begin{remark}[Higher roots]
The idea of isolating a factor $\Lambda^2$ from the product over $\varepsilon\in\{-1,1\}^K$ is also used by Dubickas \cite{Dubickas_square_roots}. The second moment argument leading to Proposition \ref{old thm}, combined with Dubickas' argument, also applies to nonzero integer linear combinations of $(a_n^{1/\ell})_{n\leq K}$ for any integer $\ell\geq3$. In this case, the Rademacher variables are replaced by i.i.d. random variables uniformly distributed on the $\ell$-th roots of unity; we leave the details to the interested reader. The additional factor of $e^{(2^{K-1}-1)/K}$ obtained here uses the real signs and further moment estimates.
\end{remark}

\noindent\textbf{Acknowledgements.}
The research of Marco Aymone is funded by FAPEMIG grant Universal APQ-00256-23. Samuel Figueredo is funded by the Coordena\c c\~ao de Aperfei\c coamento de Pessoal de N\'ivel Superior - Brazil (CAPES) - finance code 001. Christian T\'afula is funded by the S\~ao Paulo Research Foundation (FAPESP), Brazil, Process No. 2025/15961-3.

\noindent\textbf{LLM assistance.}
ChatGPT, Claude and Gemini were used to obtain the factor of $e^{(2^{K-1}-1)/K}$ in Theorem \ref{theorem principal}, and to develop and revise parts of the proofs. The authors are responsible for the mathematical content and its presentation.

\bigskip
\noindent\small
Departamento de Matem\'atica, Universidade Federal de Minas Gerais, Av. Ant\^onio Carlos, 6627, CEP 31270-901, Belo Horizonte, MG, Brazil.\\ Email addresses: \texttt{aymone.marco@gmail.com} and \texttt{sn.figueredo2@gmail.com}
\medskip

\noindent
School of Mathematics and Statistics, University of New South Wales, Kensingtion NSW, 2033, Australia\\
Email address: \texttt{siddharth.iyer@unsw.edu.au} and \texttt{siddharthiyer@live.com}

\noindent
Instituto de Matem\'atica, Estat\'istica, e Ci\^encia da Computa\c c\~ao, Universidade de S\~ao Paulo, Rua do Mat\~ao, 1010, CEP 05508-090, S\~ao Paulo, SP, Brazil.\\ Email address: \texttt{tafula@ime.usp.br}
\medskip
\end{document}